\documentclass[11pt,hyp,]{nyjm}
\usepackage{hyperref}
\hypersetup{nesting=true,debug=true,naturalnames=true}
\usepackage{graphicx,upref}
\usepackage{amsfonts,amssymb,amsthm,epsfig,epstopdf,url,array}
\usepackage{amsthm}
\usepackage[retainorgcmds]{IEEEtrantools}
\usepackage{upgreek}
\usepackage{latexsym}
\let\<\langle
\let\>\rangle

\let\uml\"

\theoremstyle{plain}
\newtheorem{thm}{Theorem}[section]
\newtheorem{lem}[thm]{Lemma}

\newtheorem{cor}{Corollary}
\newtheorem{Claim}{Claim}

\theoremstyle{plain}
\newtheorem{defn}{Definition}[section]

\theoremstyle{definition}

\newcommand{\Bf}[1]{{\bf{#1}}}
\newcommand{\V}{{\bf{v}}}
\newcommand{\X}{{\bf{x}}}

\newcommand{\U}{{\bf{u}}}

\newcommand{\W}{{\bf{w}}}

\newcommand{\F}{{\bf{f}}}

\title{Nonconvection and uniqueness in Navier-Stokes equation} 

\author{Waleed S. Khedr}  
\address{Cairo, Egypt} 
\email{waleedshawki@yahoo.com}  


\begin{document} 
 
\begin{abstract}  
      % Abstracts are required.
      In the presence of a certain class of functions we show that there exists a smooth solution to Navier-Stokes equation. This solution entertains the property of being nonconvective. We introduce a definition for any possible solution to the problem with minimum assumptions on the existence and the regularity of such solution. Then we prove that the proposed class of functions represents the unique solution to the problem and consequently we conclude that there exists no convective solutions to the problem in the sense of the given definition.
\end{abstract} 
\maketitle
\tableofcontents

%%%% **** The text of the paper starts here **** %%%%
\section{Introduction}
Navier-Stokes equation was named after the French engineer Navier who was the first to propose this model. This model was investigated later by Poisson and de Saint Venant. However, Stokes was the one who justified the model based on the principles of continuum mechanics. By advent of 1930 the interest in this model increased significantly and outstanding results were obtained by Leray, Hopf, Ladyzhenskaya and Finn.\\

In continuum mechanics the classification of the flow modelled by any equation is a consequence of the constitutive assumptions on the fluid. Newtonian fluids are those exhibiting the property of being viscous. Viscosity of the fluid leads to the presence of frictional forces which in turn induces frictional stress. The Navier-Stokes equation models the flow of a Newtonian fluid and it can be derived by employing Newton's second law, the balance of momentum and the law of conservation of mass. Such employment yields an equation of motion in terms of the stress tensor which, in the case of Navier-Stokes equation, is linear in terms of the deformation gradient. This model is very important in a variety of Physical, Mechanical, Engineering and Mathematical applications. Its applications in Oil and Gas industry is remarkable. Also it can be employed in Biology to represent blood's or nutrition's flow in a body. It plays a great role in the design of aerodynamics. For a suitable physical background the reader is advised to review \cite{gurtin,kambeb}.\\

This model poses a serious challenge when it comes to proving the existence and the smoothness of its solution. This problem was perfectly addressed by Ladyzhenskaya in two dimensional space among many other issues in higher dimensional spaces \cite{ladyF1}. However, a decisive answer in the three dimensional space or higher remains unavailable. It is almost impossible to enlist all the results obtained for this equation. Therefore, we suggest for the interested reader to review the monographs \cite{ladyF1,galdi,majda} and the references within for much more details.\\

Recently, the interest in this equation is not fading at all. There are persistent efforts to clarify the properties of the solution, especially its smoothness. Among many respectful results, we mention the outstanding analysis by Tao in \cite{tao}, the work of Constantin in \cite{consta1,consta2, consta3}. A very interesting result for partial regularity of suitable weak solutions was obtained by Caffarelli in \cite{caf}.\\

In \cite{khedrN} we proposed a class of non-convective solutions to Navier-Stokes equation and we proved that this class does actually represent the unique classical solution to the problem. We deduced that the flow represented by such solution is tangential to the boundary in the case of a bounded domain. We also demonstrated the global existence of this solution both in space and time in unbounded domains, and we derived rough estimates for the associated rates of decay. Most importantly, the uniqueness argument introduced in \cite{khedrN} implied that there is no classical solution for Navier-Stokes equation unless it is non-convective.\\

In this article we introduce a definition for acceptable solutions to Navier-Stokes equation. We choose as minimum assumptions on the solution's existence and regularity as possible so that we can generalize our results to any possible solution to the problem. Our main result is that Navier-Stokes equation is nonconvective such that one can safely drop the convective term and reduce the problem to a linear model.\\
 
To this end, a class of nonconvective solutions is proposed and discussed in the sense of being a representative for acceptable solutions. Then it is proved that this class represents the unique solution to the problem, which naturally implies the absence of convection in Navier-Stokes equation and the linearity of the model. Most of the results are obtained by considering standard theories of partial differential equations. In the next section a statement of the problem is introduced along with some definitions, notations and the employed functional spaces. Afterwards, the proofs of the main results are established. 
%%%%%%%%%%%%%%%%%%%%%%%%%%%%%%%%%%%%%%%%%
\section{Statement of the problem}
\label{statement}
The spatial domain is $\Omega$ which is either a bounded region in $\mathbb{R}^N$ or the whole of $\mathbb{R}^N$. For the sake of conciseness, we use the notation $\Omega_t$ to denote $\Omega_t=\{(\X,t):\X\in\Omega, t\in(0,\infty)\}$. Clearly, such notation should not be taken to imply a moving boundary.\\

The well known Lebesgue spaces $L^q(\Omega)$ are used to represent the functions with bounded mean of order $q$. We also use the Sobolev space $H^m(\Omega)$ to represent functions with bounded derivatives such that for a vector field $\V=\{v_1,\dots,v_N\}$ we have $\partial^{|\alpha|}v_i\in L^2(\Omega)$ for every $|\alpha|=1,\dots,m$ and $i=1,\dots,N$.\\

This motivates the usage of the spaces $V^m(\Omega)$ and $V^m_0(\Omega)$, which are well known spaces of functions in the theory of incompressible fluids as representatives for divergence free (solenoidal) bounded vector fields such that $V^m(\Omega)=\{\V\in H^m(\Omega):\nabla\cdot\V=0\;\text{in}\;\Omega\}$ and $V^m_0(\Omega)=\{\V\in H^m(\Omega):\nabla\cdot\V=0\;\text{in}\;\Omega,\;\V=0\;\text{on}\;\partial\Omega\}$\\

In the following model $\V_0$ will be used to denote the initial profile, $\V^*$ will denote the boundary datum, $\F$ denotes the total of the external body forces and $E(t)$ represents the Kinetic energy of the flow. The smoothness of $\V_0(\X)$ is such that 
\begin{equation}
\V_0(\X)\in\Bf{C}^2(\overline{\Omega})\cap V^{N+2}(\Omega).
\label{initialdata}
\end{equation}
The smoothness of the boundary datum $\V^*(\X_{N-1},t)$ is such that 
\begin{equation}
\V^*(\X_{N-1},t)\in\Bf{C}^\infty(\partial\Omega_t)\;\text{and}\; \V^*(\cdot,t)\sim t^{-K^*}\;\text{for any}\;K^*>1.
\label{bounddata}
\end{equation}
Finally, the forcing term $\F$ is smooth in space and time such that
\begin{equation}
\F(\X,t)\in\Bf{C}^1([0,\infty]; \Bf{C}^1(\overline{\Omega})\cap {\color{black}H^1}(\Omega))\;\text{and}\; \F(\cdot,t)\sim t^{-K}\;\text{for any}\;K>0.
\label{forcedata}
\end{equation}
By laws of classical mechanics, the energy generated by a moving object is proportional to the square of its velocity. Hence, the energy $E(t)$ generated by the flow $\V$ is defined as follows
\begin{equation}
E(t)=\int_\Omega|\V(\X,t)|^2d\X.
\label{energyID}
\end{equation}
Recall that the above integral represents the norm of $\V$ in the Lebesgue space $L^2(\Omega)$.\\
The main Model Equation is in the form
\begin{equation}
\left\{
 \begin{IEEEeqnarraybox}[
 \IEEEeqnarraystrutmode
 \IEEEeqnarraystrutsizeadd{5pt}
 {2pt}
 ][c]{l}
 \Bf{v}_{t}+(\Bf{v}\cdot\nabla)\Bf{v}=\mu\Delta \Bf{v}-\nabla p+\F\quad\text{in}\quad\Omega_{t},
 \\
 \upomega=\nabla\times\V\quad\text{in}\quad\Omega_t,
 \\
 \Bf{v}(\Bf{x},0)=\Bf{v}_0(\Bf{x})\quad\text{and}\quad\upomega_0(\X)=\nabla\times\V_0(\X)\quad\text{in}\quad\Omega,
 \\
 \Bf{v}(\Bf{x},t)=\V^*(\X_{N-1},t)\quad\text{on}\quad\partial\Omega_t,
 \\
 \nabla\cdot \Bf{v}=\nabla\cdot\upomega=0\quad\text{in}\quad\overline{\Omega}_t
 \end{IEEEeqnarraybox}
\right.
\label{vmodel}
\end{equation}
where the last equation in the above model is what many authors commonly refer to as the {\bf incompressibility condition} or the {\bf solenoidal condition}. The first term in the first equation is the acceleration of the fluid's flow {\color{black}in time}, the second is the convective term {\color{black}that represents the acceleration of the flow in space}, the third represents the diffusion scaled by the kinematic viscosity constant $\mu$, the fourth is the pressure, and the last one represents the total of the external body forces. The solution $\Bf{v}$ is the vector field representing the velocity of the flow in each direction, and its rotation $\upomega$ is the vorticity\textbf{}. Note that $\nabla\cdot\upomega=0$ in $\overline{\Omega}_t$ by compatibility.\\

The last task in this section is to introduce a definition to any possible solution to Model Equation \eqref{vmodel}. To this end, dot product Model Equation \eqref{vmodel} by $\V$, assume an arbitrary domain $\Omega$, integrate by parts over $\Omega$, employ the Divergence Theorem and, without loss of generality, assume that $\V$ vanishes rapidly enough as $|\X|\rightarrow\infty$. This should yield a basic energy estimate in the form
\begin{equation}
\frac{1}{2}\frac{d}{dt}\int_\Omega|\V(\X,t)|^2d\X+\mu\int_{\Omega}|\nabla\V|^2d\X=\int_{\Omega}\F\cdot\V\,d\X.
\label{basicid}
\end{equation}
In light of the assumptions on $\V_0$ and $\F$ one can readily conclude that $\V,\nabla\V\in L^2(\Omega)$ for every $t>0$. However, to serve the purpose of this study, we will consider the minimum assumptions on the regularity of $\V$ so that any possible solution to Model Equation \eqref{vmodel} cannot be defined in a weaker sense.
\begin{defn}
We say that $\V(\X,t)$ is a possible solution to Model Equation \eqref{vmodel} if $\V\in L^2(\Omega)$ for every $t>0$.
\label{defv}
\end{defn} 
In light of the above definition; the investigation on the existence of a nonconvective, smooth and unique solution to Model Equation \eqref{vmodel} will be established. 
%%%%%%%%%%%%%%%%%%%%%%%%%%%%%%%%%%%%%%%%%
\section{Main results}
\label{mainres} 
In this part we will follow the same procedure introduced in \cite{khedrN} by proposing a certain class of functions to serve as a candidate for possible solutions to Model Equation \eqref{vmodel}. Afterwards we investigate its validity and the consequences of that choice. The proposed class of solutions is given by the following claim.
\begin{Claim}
Any solution to Model Equations \eqref{vmodel}, in the sense of Definition \ref{defv}, is nonconvective and it takes the form $\V(\X,t)=\psi(\X,t)\U(t)$ where $\psi:\mathbb{R}^N\times\mathbb{R}\rightarrow\mathbb{R}$ is a scalar field and $\Bf{u}=(u_1(t),\dots,u_N(t))$ is a vector field independent of $\X$. 
\label{claim1}
\end{Claim}

An important question in the theory of Navier-Stokes equation is the ability to verify the compatibility condition on the boundary with minimum restrictions on the flux passing through the boundary especially if $\partial\Omega$ is divided into several parts \cite[pp. 4-8]{galdi}. This condition is a natural consequence of the incompressibility of the flow. Hence, it takes the form
\begin{equation}
\int_{\partial\Omega}(\Bf{v}\cdot\vec{\Bf{n}})d\X_{N-1}=0,
\label{boundarycond}
\end{equation}
where $\vec{\Bf{n}}$ is the outward unit vector normal to $\partial\Omega$. The class of solutions proposed in Claim \ref{claim1} provides an answer to the problem of compatibility on the boundary. This answer is independent of the equation being investigated as shown by the following lemma.
%%%%%%%%%%%%%%%%%%%%%%%%%%%
\begin{lem}[Tangential flow]
Let $\Omega$ be an arbitrary domain, $\psi(\X,t):\mathbb{R}^N\times\mathbb{R}\rightarrow\mathbb{R}$ be any scalar field such that $\psi\in\Bf{C}^1(\overline{\Omega})$ and let $\U(t)=(u_1(t),\dots,u_N(t))$ be any vector field independent of $\X$. The Compatibility Condition \eqref{boundarycond} is satisfied for every divergence free vector field $\V(\X,t):\mathbb{R}^N\times\mathbb{R}\rightarrow\mathbb{R}^N$ in the form $\V=\psi(\X,t)\U(t)$. In particular, on every part of $\partial\Omega$, $\V$ and its rotation $\upomega$ are tangents to $\partial\Omega$ such that $\V\cdot\vec{\Bf{n}}=\upomega\cdot\vec{\Bf{n}}=0$.
\label{compcond}
\end{lem}
%%%%%%%%%%%%%%%%%%%%%%%%%%%%%%%%%%%%%%%%%
\begin{proof}
The proof is identical to the proof of \cite[Lemma 1]{khedrN}.
\end{proof}
%%%%%%%%%%%%%%%%%%%%%%%%%%%%%%%%%%%%%%%
The last proof was given independent of any equation and for an arbitrary domain. This means that the concluded property of being tangential to the boundary is entertained by this class of solutions regardless the model. Due to the incompressibility of the fluid the equation of the conservation of mass is transformed into a simple solenoidal identity $\nabla\cdot\V=0$. Hence, the proposed class of solutions is naturally nonconvective when considered as a representative for any incompressible fluid.\\

This class of solutions entertains another important property that will be of a great significance in proving the main result of this study. For simplicity, assume the absence of the forcing term $\F$. In this case, Identity \eqref{basicid} shows that the dissipation of the energy of the motion depends directly on the norm of the tensor $\nabla\V$. Since $\nabla\cdot\V=0$ in $\Omega$, then $\rm{tr}(\nabla\V)=\sum_i\lambda_i=0$ where $\lambda_i$ are the eigen values of the tensor $\nabla\V$. This implies that either all $\lambda_i=0$ or there exist at least one positive and one negative eigen values to the tensor $\nabla\V$. If we interpret the eigen values of $\nabla\V$ as weights, what would be the result of applying these weights to the energy integral? The answer to this question is addressed by the statement of the following lemma. 
%%%%%%%%%%%%%%%%%%%%%%%%%%%%%%%5
\begin{lem}
If $\V$ is in the form proposed in claim \ref{claim1} or $\V=0$ on $\partial\Omega$, then the energy integral \eqref{energyID} when weighted by the eigen values of the tensor $\nabla\V$ becomes identically zero.
\label{eigenlem}
\end{lem}
\begin{proof}
The eigen value representation of the tensor $\nabla\V$ applied to the vector $\V$ takes the form
\[
\nabla\V\V=\lambda_\V\V.
\]
If the solution is in the form $\V=\psi\U$, then $\nabla\V\V=0$ and the result follows immediately. Now, assume a general solution $\V$ that vanishes on $\partial\Omega$. Dot product the above equation by $\V$, integrate by parts over $\Omega$ and apply the Divergence Theorem to obtain
\[
0=\int_{\partial\Omega}|\V|^2\V\cdot\vec{\Bf{n}}\,d\X_{N-1}=\int_\Omega\lambda_\V|\V|^2d\X,
\]
where we used {\it Lemma \ref{compcond}} and this concludes the proof.
\end{proof}
%%%%%%%%%%%%%%%%%%%%%%%%%%%%%%%%%   
Now, one needs to check that a solution to Model Equation \eqref{vmodel} represented by the proposed class is physically and mathematically acceptable. This is addressed by the statement of the following theorem.
%%%%%%%%%%%%%%%%%%%%%%%%%%%%%%%
\begin{thm}[Existence and smoothness]
Let $\Omega\subseteq\mathbb{R}^N$ be any domain with sufficiently smooth boundary $\partial\Omega$ if it is bounded. Let $\Omega_t=\Omega\times(0,\infty)$. Suppose that $\Bf{v}_0(\Bf{x})$, $\V^*(\X_{N-1},t)$ and $\F(\X,t)$ satisfy Conditions \eqref{initialdata}, \eqref{bounddata} and \eqref{forcedata} respectively. If $\V(\X,t)$ is in the form proposed in {\it Claim \ref{claim1}}, then Model Equation \eqref{vmodel} has a smooth solution $(\V,\upomega,p)$ with bounded energy $E(t)$ such that $\Bf{v}(\Bf{x},t)\in \Bf{C}^\infty(\overline{\Omega}_t)$, $\upomega(\X,t)\in\Bf{C}^\infty(\overline{\Omega}_t)$ and $p(\X,t)\in C^1([0,\infty];C^3(\overline{\Omega}))$. In particular, the exact solution is given by solving the following system
\begin{equation}
\left\{
 \begin{IEEEeqnarraybox}[
 \IEEEeqnarraystrutmode
 \IEEEeqnarraystrutsizeadd{5pt}
 {2pt}
 ][c]{l}
 \left.
  \begin{IEEEeqnarraybox}[
 \IEEEeqnarraystrutmode
 \IEEEeqnarraystrutsizeadd{5pt}
 {2pt}
 ][c]{l}
 \V_{t}-\mu\Delta\V=-\nabla p+\F
 \\
 \Delta p=\nabla\cdot\F
 \end{IEEEeqnarraybox}
 \right\}\quad\text{in}\quad\Omega_{t},
 \\
 \V(\X,t)=\V^*(\X_{N-1},t)\quad\text{on}\quad\partial\Omega_t,
  \\
 \V(\X,0)=\V_0(\X)\quad\text{in}\quad\overline{\Omega},
 \\
 \nabla\cdot \Bf{v}=0\quad\text{in}\quad\overline{\Omega}_t
 \end{IEEEeqnarraybox}
\right.
\label{bsol1}
\end{equation}
where $\nabla p\cdot\vec{\Bf{n}}$ can be defined uniquely in terms of the values of $\V$ and $\F$ on the boundary. Moreover, if $\F\in \Bf{C}^\infty(\overline{\Omega}_t)$ then $p\in C^\infty(\overline{\Omega}_t)$.
\label{existthm}
\end{thm}  
\begin{proof}
If $\V$ is in the form $\V=\psi(\X,t)\U(t)$, then $(\V\cdot\nabla)\V=0$. It follows that Model Equation \eqref{vmodel} is reduced to
\begin{equation}
\V_t-\mu\Delta\V=-\nabla p+\F.
\label{step1}
\end{equation}
Apply the divergence operator and recall that $\nabla\cdot\V=\nabla\cdot\Delta\V=0$ by incompressibility to obtain
\begin{equation}
\Delta p=\nabla\cdot\F.
\label{step2}
\end{equation}
Note that Equation \eqref{step1} is a standard heat equation and that Equation \eqref{step2} is a standard Poisson  equation. By the standard theory of second order elliptic equations and Condition \eqref{forcedata} one concludes that $p\in C^1([0,\infty]; C^3(\overline{\Omega}))$, and if $\F\in \Bf{C}^\infty(\overline{\Omega}_t)$, then $p\in C^\infty(\overline{\Omega}_t)$. Actually, one can express $p$ fundamentally in terms of the Newtonian potential, see \cite{majda,evans,fritz}. Note that we can calculate $\nabla p$ on the boundary $\partial\Omega$ in terms of the values of $\V$ and $\F$ on $\partial\Omega$. Having $\nabla p\cdot\vec{\Bf{n}}$ as a form of boundary conditions for $p$ on the boundary guarantees its uniqueness up to a constant even if $\F$ is not introduced, see \cite{fritz}.Using the solution $p$ one can solve for $\V$. Whether $\F$ is present or not, Equation \eqref{step1} is a standard heat equation for which the solution can be expressed fundamentally in terms of the Gaussian kernel and the given data. The continuity of the given data and the smoothing property of the Gaussian kernel ensures the infinite differentiability of $\V$. The uniqueness of $\V$ follows by the presence of the boundary datum $\V^*$ which can be inserted using Dirichlet's Green function . By definition $\upomega=\nabla\times\V$, which implies immediately that $\upomega\in\Bf{C}^\infty(\overline{\Omega}_t)$, and it implies its uniqueness as well.\\

When the domain is bounded one can easily show the boundedness of $\V$ and $\nabla\V$ in $L^2(\Omega)$ for every $t>0$ due to the assumptions on the data and the boundedness of the domain itself. In the case of unbounded domain, one simply uses Identity \eqref{basicid} to conclude the boundedness of $\V$ and $\nabla\V$ in $L^2(\Omega)$ for every $t>0$. This implies the boundedness of the energy $E(t)$.
\end{proof}
%%%%%%%%%%%%%%%%%%%%%%%%%%%%%%%%%%%%%%
It remains to answer the question whether there exists another possible solution to Model Equation \eqref{vmodel} in the sense of Definition \ref{defv}. The orthogonality argument presented in \cite[Theorem 4]{khedrN} to prove uniqueness remains valid in the context of this article. It simply depends on the coincidence of any possible solutions on the boundary. However, the proof of the following theorem will be introduced in a different manner, in which no prior knowledge of the continuity of the solution is required. Simple energy estimates will be employed and the result follows by adopting the minimum assumptions of Definition \ref{defv}.
%%%%%%%%%%%%%%%%%%%%%%%%%%%%%%%%%%%%%%%%
\begin{thm}[Uniqueness]
Let $\Omega$, $\V_0$ and $\F$ be chosen arbitrarily. Any possible solution to Model Equation \eqref{vmodel}, in the sense of Definition \ref{defv}, is in the form $\V=\psi(\X,t)\U(t)$. In particular, it is non-convective. 
\label{mainthmuniq}
\end{thm}
\begin{proof}
Let $\V$ denotes the proposed nonconvective solution. By virtue of {\it Theorem \ref{existthm}} it is clear that $\V$ satisfies the conditions of Definition \ref{defv}. Assume there is a convective solution $\V^g$ that satisfies the conditions of Definition \ref{defv} as well, and for which $(\V^g\cdot\nabla)\V^g$ exists. Now, let $\W=\V-\V^g$ so that the difference equation takes the form
\begin{equation}
\W_t-\mu\Delta\W+\nabla q=(\V^g\cdot\nabla)\V^g,
\label{diffequ1}
\end{equation} 
where $\W(\X,0)=\W(\partial\Omega,t)=0$. Recall that $\nabla\cdot\W=0$ by the incompressibility of $\V$ and $\V^g$ and also $\W\in L^2(\Omega)$ for every $t>0$ by definition. Dot product Equation \eqref{diffequ1} by $\W$, integrate over $\Omega$ and apply the Divergence Theorem to obtain
\begin{eqnarray}
\frac{1}{2}\frac{d}{dt}\int_\Omega|\W|^2d\X+\mu\int_\Omega|\nabla\W|^2d\X&=&\int_\Omega(\nabla\V^g)^T\W\cdot\V^g\,d\X.\nonumber\\
&=&\int_\Omega\left(\nabla(\V^g\cdot\W)-(\nabla\W)^T\V^g\right)\cdot\V^g\,d\X\nonumber\\
&=&-\int_\Omega(\nabla\W)^T\V^g\cdot\V^g\,d\X\nonumber\\
&=&-\int_\Omega(\nabla\W)\V^g\cdot\V^g\,d\X.
\label{diffequ1*}
\end{eqnarray}
The second term in the left hand side can be ignored by positivity to obtain
\begin{eqnarray}
\frac{1}{2}\frac{d}{dt}\int_\Omega|\W|^2d\X&\leq&-\int_\Omega\lambda_\W|\V^g|^2d\X.
\label{diffequ1**}
\end{eqnarray}
Now consider the following argument: since $\nabla\W\W=\lambda_\W\W$ and $\W=0$ on $\partial\Omega$, then by using Lemma \ref{eigenlem} one obtains
\begin{eqnarray}
0=\int_{\partial\Omega}|\W|^2\W\cdot\vec{\Bf{n}}\,d\X_{N-1}=\int_\Omega\lambda_\W|\W|^2d\X.
\label{diffeque2}
\end{eqnarray}
One can write also $\nabla\W\V=\lambda_\W\V$, which upon dot product by $\V$ yields
\[
\nabla\cdot((\W\cdot\V)\V)=\lambda_\W|\V|^2,
\]
where we used the fact that $\nabla\V\V=0$. Integrate over $\Omega$, recall that $\W=0$ on $\partial\Omega$, and also that $\V\cdot\vec{\Bf{n}}=0$ to reach
\begin{eqnarray}
0=\int_{\partial\Omega}(\W\cdot\V)\V\cdot\vec{\Bf{n}}\,d\X_{N-1}=\int_\Omega\lambda_\W|\V|^2d\X.
\label{diffeque3}
\end{eqnarray}
Finally, in the same fashion, one can obtain the identity 
\begin{eqnarray}
0=\int_{\partial\Omega}|\W|^2\V\cdot\vec{\Bf{n}}\,d\X_{N-1}=\int_\Omega\lambda_\W(\V\cdot\W) d\X.
\label{diffeque4}
\end{eqnarray}
Since $|\V^g|^2=\V^g\cdot\V^g=|\V|^2-2\V\cdot\W+|\W|^2$, then by substituting Identities \eqref{diffeque2}, \eqref{diffeque3} and \eqref{diffeque4} in Identity \eqref{diffequ1**} one obtains
\[
\frac{d}{dt}\|\W\|_{L^2(\Omega)}\leq0,
\]
and since $\W(\X,0)=0$, then $\W=0$ almost everywhere. By Identity \eqref{diffequ1*} one concludes also that $\nabla\W=0$ almost everywhere. This readily implies that $\V^g\cdot\nabla\V^g=0$ almost everywhere. If $N=2$, then by the Embedding Theorem, $\W$ is H\"older continuous and consequently $\W=0$ everywhere. If $N>2$, one can extend the same result to everywhere by differentiating Equation \eqref{diffequ1} and using the fact that $\V^g\cdot\nabla\V^g=0$ to conclude the vanishing of higher derivatives of $\W$ almost everywhere. The Embedding Theorem then can be applied to deduce the continuity of $\W$ and consequently that the convective term $\V^g\cdot\nabla\V^g=0$ everywhere in $\Omega_t$. This concludes the proof.
\end{proof}
%%%%%%%%%%%%%%%%%%%%%%%%%%%%%%%%%%%%%%%%
\begin{cor}
Claim \ref{claim1} is true.
\end{cor}
%%%%%%%%%%%%%%%%%%%%%%%%%%%%%%%%%%%%%%%%%%
\section{Conclusions}
A class of nonconvective solutions was proposed as a candidate for solutions to Navier-Stokes equation. In an approach distant from the one introduced in \cite{khedrN}, we proved in this article that the proposed solution is a very smooth solution and we provided another form of solvability. \\

Most importantly, we proved the uniqueness of the proposed solution as the only possible solution to Navier-Stokes equation. To this end and for the sake of generalization, we introduced a definition for any possible solution to the problem. Then we considered an argument that includes the weakest possible solution to the problem. This form of uniqueness implies that Navier-Stokes equation is not convective and it can be reduced to a linear model.  
%%%%%%%%%%%%%%%%%%%%%%%%%%%%%%%%%%%%%%%%%%%%%
\def\cprime{$'$} \def\cprime{$'$}

\end{document}